\documentclass[11pt,twoside]{article}
\usepackage{a4}
\usepackage{amssymb,amsmath,amsthm,latexsym}
\usepackage{amsfonts}
\usepackage{amsfonts}
\usepackage{url}
\usepackage{graphicx}
\newtheorem{theorem}{Theorem}[section]

\newtheorem{lemma} [theorem]{Lemma}

\newtheorem{remark}[theorem]{Remark}

\vspace{5cm}

\begin{document}
	
	\label{'ubf'}  
	\setcounter{page}{1}                                 

	\markboth {\hspace*{-9mm} \centerline{\footnotesize \sc
			Theta function identities and $q$ series }
	}
	{ \centerline                           {\footnotesize \sc  
			Hemant Masal, Subhash Kendre, Hemant Bhate                           } \hspace*{-9mm}              
	}

	\vspace*{-2cm}

	\begin{center}
		{ 
			{\Large \textbf { \sc Theta function identities and $q$ series
				}
			}
			\\

			\medskip

			{\sc Hemant Masal, Subhash Kendre,Hemant Bhate }\\
			{\footnotesize hemantmasal@gmail.com}\\
			{\footnotesize sdkendre@yahoo.com}\\
			{\footnotesize bhatehemant@gmail.com}
		}\\
		
	\end{center}

	\thispagestyle{empty}

	\hrulefill

	\begin{abstract}  
		{\footnotesize We establish some functional identities of theta functions, an elementary proof of classical fourth-order identities, Landen transformations, and q series from the eigenvectors of the discrete Fourier transform. Also, we derive connection between Rogers-Ramanujan type identity and theta function identity.
		}
	\end{abstract}
	\hrulefill
	

	
	\section{Introduction}
	
The Discrete Fourier Transform (DFT) of degree $n$ can be represented by a unitary matrix \cite{mehta1987eigenvalues,terras1999Fourier} $A$ whose entries  are given by $A_{jk}=\frac{1}{\sqrt{n}}e^{\frac{2\pi i}{n}jk}$. The problem of diagonalizing the DFT arises in many contexts and has been studied extensively. In Section III of \cite{mehta1987eigenvalues}, the spectral decomposition of the DFT matrix is explained and it gives rise to many relations between spectral projectors, the multiplicity of eigenvalues of the DFT, and multiplicities of eigenvalues of spectral projectors. In addition, the eigenvectors of the DFT which involve Hermite functions are obtained.

\par Matveev \cite{matveev2001intertwining} gives the spectral decomposition of a bounded operator $U$ on a Hilbert space which is a root of the identity, (i.e.$ U^n=I$). Note that $A^4=I,$ and $A$ is unitary. Matveev applies the spectral decomposition to the DFT matrix. 
\par The generalized theta function \cite{matveev2001intertwining,mckean1999elliptic} is defined by 
\begin{equation}
	\theta(x, \tau,\nu)=\sum_{m=-\infty}^{\infty}e^{\pi i \tau m^{2\nu}+2\pi imx},~\nu \in {\mathbb{Z}}^+, ~Im~\tau >0.
\end{equation}

and the generalized theta function with characteristics $(a,b)$ is 
\begin{equation}
	\theta_{a,b}(x, \tau,\nu)=\sum_{m=-\infty}^{\infty}e^{\pi i \tau (m+a)^{2\nu}+2\pi i(m+a)(x+b)},\quad
	\nu \in {\mathbb{Z}}^+, ~~~Im~\tau >0.
\end{equation}
Matveev \cite{matveev2001intertwining} shows that these theta functions and generalized theta functions give a rise to a set of eigenvectors of the DFT. 
\par Identities between theta functions are classically proved using their translational properties, location and multiplicities of zeros, and periodicity or quasi-periodicity properties. This is the method used for example in \cite{mckean1999elliptic} to prove null identities and the duplication formula. Similar ideas have been extended in \cite{cooper2008determinant} to obtain determinant identities for Theta functions. In \cite{cooper2008determinant}, it is shown that the Gosper-Schroeppel identity can be obtained either by translational properties of theta functions or as a consequence of the fact that $\det[(r_j s_k + t_j u_k)_{1~ \leq~ j,k ~\leq ~3}]=0.$ These ideas were further extended in \cite{toh2008generalized} to obtain determinant identities for m-th order Theta functions.
\par The classical Watson's identity and the  Riemann identity are consequences of the eigenvalues and eigenvectors of the DFT \cite{malekar2010discrete,malekar2012discrete}. 
\par In this article, the facts that theta functions with characteristics give rise to the eigenvectors of the DFT and that the corresponding eigenvalues are non-degenerate are exploited to obtain new identities for theta functions, Landen transformations and a new elementary proof of classical fourth-order identities. These are in fact determinantal identities, the determinant being of a submatrix of the matrix of the eigenvectors of the DFT.
\par This method also gives new identities, for example, the following results are obtained:
\begin{itemize} 
	
	\item (Lemma \ref{L1}) Order one (linear) relation
	\begin{equation*}
		\begin{split}
			& \sqrt{2}\theta(2x, 4\tau)+\theta (x, \tau)= (\sqrt{2}+1)\left(\sqrt{2}\theta_{1/2,0}(2x, 4\tau)+\theta_{0, 1/2} (x, \tau )\right).
		\end{split}
	\end{equation*}

	\item Landen type transformation theta functions are also introduced. For example, the following result is derived in Lemma \ref{ldn}, 
	\begin{equation*}
		\begin{split}
			\big[ \theta(2x, 8\tau) + \theta_{1/2,0}(2x, 8\tau) \big]^2 - \big[ \theta_{1/4,0}(2x, 8\tau) + \theta_{3/4,0}(2x, 8\tau) \big]^2 
			=\theta_{0,1/2}(x, \tau)\theta_{0,1/2}(0, \tau).
		\end{split}		
	\end{equation*}
	\item In lemma \ref{R-R}, it is shown that the Rogers-Ramanujan type identity 
	\begin{equation*}
		\begin{aligned}
			\prod_{n=1}^{\infty} \left( 1-q^{4n}\right) \left( 1+q^{4n-3}z^2 \right) \left( 1+q^{4n-1}z^{-2} \right)& +  z \prod_{n=1}^{\infty} \left( 1-q^{4n}\right) \left( 1+q^{4n-1}z^2 \right) \left( 1+q^{4n-3}z^{-2} \right)\\
			&= \prod_{n=1}^{\infty} \left( 1-q^{n}\right) \left( 1+q^{n-1}z \right) \left( 1+q^{n}z^{-1} \right).
		\end{aligned}
	\end{equation*}
	corresponds to the first theta function identity obtained in lemma \ref{s3},
	
	\item  Following $q$ identity is given in equation (\ref{fq})
	\begin{equation}
		\begin{split}
			\sum_{m=-\infty}^{\infty}q^{m(2m+1)}
			=\frac{(q^{2};q^{2})^2_{\infty}[(-q;q^{2})^2_{\infty}+2q^{1/4}(-q^{2};q^{2})^2_{\infty}]^2-(q^{1/4};q^{1/4})^2_{\infty}(q^{1/8};q^{1/4})^4_{\infty}}{4q^{1/8}(q^{4};q^{4})_{\infty}[(-q^{2};q^{4})^2_{\infty}+2q^{1/2}(-q^{4};q^{4})^2_{\infty}]}.
		\end{split}
	\end{equation}
\end{itemize} 
These identities (and many more) have not appeared earlier.
\section{Eigenvalues and eigenvectors of A}

We first recall basic facts about the eigenvalues of the DFT matrix and their multiplicities, (see \cite{ mehta1987eigenvalues,matveev2001intertwining,terras1999Fourier}).
Let $n$ be fixed, and $A$ be the DFT matrix whose $(j,k)^{th}$ entry  is given by $A_{jk}=\frac{1}{\sqrt{n}}e^{\frac{2\pi i}{n}jk}.$  
\[ {A^{2}}_{jk} =
\begin{cases}
	1       & \quad \text{if } j+k \text{ =0 (mod n)}\\
	0 & \quad  \text{ otherwise}\\
\end{cases}\]

\par The eigenvalues of $A$ are $1,-1, i,-i,$ and they have non-negative multiplicity. The multiplicities are given by $[(n+4)/4], [(n+2)/4],[(n+1)/4]$ and $[(n-1)/4]$ respectively. where $[x]$ is the greatest integer not greater than $x$.
We recall the following Theorem,
\begin{theorem}(Matveev \cite{matveev2001intertwining})
	\label{matveev}
	For any $\tau$, with Im $\tau >0$ vector  $G(x, \tau,\nu, k)=G(k)$, whose $j^{th}$ component ($0\leq j\leq n-1$) is given by
	\begin{align}\label{eq:mv}\begin{aligned}
			G_j(x, \tau, \nu, k)=&\theta_{(\frac{j}{n},0)}(x, \tau, \nu)+(-1)^k\theta_{(\frac{-j}{n},0)}(x, \tau, \nu)\\
			+&\frac{1}{\sqrt{n}}\bigg[ (-i)^k\theta\bigg(\frac{j+x}{n},\frac{\tau}{n^{2\nu} },\nu \bigg)+ (-i)^{3k}\theta \bigg(\frac{x-j}{n},\frac{\tau}{n^{2\nu}},\nu \bigg)\bigg]
		\end{aligned}
	\end{align}
	is the eigenvector of the DFT corresponding to eigenvalue $i^k$
	i.e, $ AG(k)=i^kG(k).	$
\end{theorem}
\section{Some Identities}
In sections 3-5, we consider the case $\nu=1.$ i.e. $\theta(x,\tau,\nu)=\theta(x,\tau,1)=\theta(x,\tau).$ 
\par The identities of theta functions have an important and long history \cite{mckean1999elliptic}. The fact that some of the Theta function identities actually arise from multiplicities of the eigenvalues of the DFT is relatively recent. The Riemann identities and Watson's addition formula have been derived in this manner \cite{malekar2012discrete,malekar2017nu}.
\par Using Theorem [\ref{matveev}], we derive many identities between the theta functions. We will list some of these identities. 
\begin{lemma}\label{L1} The following identities are holds for $n=2.$
	\begin{enumerate}
		\item  $
		\sqrt{2}\theta(2x, 4\tau)+\theta \left( x, \tau \right)= (1+\sqrt{2})(\sqrt{2}\theta_{1/2,0}(2x,4 \tau)+\theta_{0,1/2} (x, \tau )),	$ 
		\item  $\sqrt{2}\theta(2x, 4\tau)-\theta \left( x, \tau  \right)=(1-\sqrt{2}) [ \sqrt{2}\theta_{1/2,0}(2x, 4\tau)-\theta_{0,1/2} \left( x, \tau \right) ],
		$
	\end{enumerate}
\end{lemma}
\begin{proof}
	For $n=2$, the eigenvalues of DFT matrix are  $\{ 1,-1\}$ each with  multiplicity one. The eigenvector $v_1$ corresponding to eigenvalue $1$ is given in Theorem  \ref{matveev} and the eigenvector $v_2$ also corresponds to eigenvalue 1.
	\begin{align*}
		v_1=\begin{pmatrix}
			2\theta(x, \tau)+\sqrt{2}\theta \left( \frac{x}{2}, \frac{\tau}{2^{2}}\right)\\
			2\theta_{\frac{1}{2},0}(x, \tau)+\sqrt{2}\theta \left( \frac{x+1}{2}, \frac{\tau}{2^{2}}  \right)
		\end{pmatrix},
		v_{2}=\frac{1}{\sqrt{2}} \begin{pmatrix}
			1+\sqrt{2} \\
			1\\
		\end{pmatrix}.
	\end{align*}
	\par 	As the eigenvalue $1$ is non-degenerate,the eigenvectors $v_1, v_2$ must be linearly dependent. It is clear that, det$[v_1,v_2]=0$.
Therefore, 
	\begin{align}
		\begin{split}
			\frac{\sqrt{2}\theta(x, \tau)+\theta \left( \frac{x}{2}, \frac{\tau}{2^{2}} \right)}{\sqrt{2}\theta_{\frac{1}{2},0}(x, \tau)+\theta \left( \frac{x+1}{2}, \frac{\tau}{2^{2}} \right)}=\sqrt{2}+1.
		\end{split}
	\end{align}
	The second equation is obtained from the non-degeneracy of the eigenvalue $-1$.
\end{proof}

\begin{lemma} \label{s3} Following identities hold for $n=4.$
	\begin{enumerate}
		\item $2[\theta(4x, 16\tau)+\theta_{\frac{1}{2},0}(4x, 16\tau)]= \theta \big( x, \tau \big)+\theta_{0,\frac{1}{2}}(x, \tau),$
		
		\item $
		2\theta(4x, 16\tau)+[\theta_{\frac{1}{4},0}(4x, 16\tau)+\theta_{\frac{3}{4},0}(4x, 16 \tau)]
		=  \theta ( x, \tau )+ 1/2[\theta_{0,\frac{1}{4}} (x, \tau)+ \theta_{0,\frac{3}{4}} (x, \tau )].
		$
	\end{enumerate}
\end{lemma}

\begin{remark}
	The first identity of the lemma \ref{s3} corresponds to the "Rogers-Ramanujan identity" (see lemma \ref{R-R}).
\end{remark}

For the DFT of order $5,6,7,8$, the eigenvalue $-i$ is non-degenerate. So, by choosing appropriate minors of order $2$ of the matrix of eigenfunctions of DFT we can have the identity given below.

\begin{lemma}
	Following identities holds for $n=5,6,7,8.$ 
	\begin{equation*}
		\begin{split}
			& [\theta_{1/n,0}(nx, n^2 \tau)-\theta_{-1/n,0}(nx, n^2 \tau)
			+\frac{1}{\sqrt{n}}[ i\theta_{0,1/n} (x,\tau) -i\theta_{0,-1/n} (x,\tau )]]\\
			=& [\theta_{2/n,0}(nx, n^2 \tau)-\theta_{-2/n,0}(nx, n^2 \tau) +\frac{1}{\sqrt{n}}[ i\theta_{0,2/n} (x,\tau )- i\theta_{0,-2/n} (x,\tau )]] \\
			&(\sin{(2\pi/n)}-{\sqrt{n}/{2}}) /\sin{(4\pi/n)}.
		\end{split}
	\end{equation*}
\end{lemma}
\begin{remark}
	All the identities proved in this section also holds for generalized theta functions $\theta(x, \tau, \nu)$ for any $\nu \in \mathbb{N}$.
\end{remark}
\section{Classical identities and Landen transformation}
In this section, we first give an alternative proof of the classical fourth-order identities of theta functions using the DFT and elementary linear algebra. We then give extensions of the Landen transformation for higher values of $n.$

\par Following identities are derived from the product of Theta functions.

\begin{lemma}\label{Lemma1}
	
	\begin{equation*}
		\begin{aligned}
			1.&~ \theta(x, \tau)\theta_{0, 1/2}(x, \tau)=\theta_{0,1/2}(2x, 2\tau)\theta_{0,1/2}(0, 2\tau),\\
			2.&~\theta_{0,1/2}(0,\tau)\theta(x,\tau)=\theta^2_{0,1/2}(x,2\tau)-\theta^2_{1/2,1/2}(x,2\tau),\\
			3.&~ \theta(x, \tau)\theta_{0,1/3}(x, \tau)=\theta_{0,1/3}(2x, 2\tau)\theta_{0,-1/3}(0, 2\tau)+\theta_{1/2,1/3}(2x, 2\tau)\theta_{1/2,-1/3}(0, 2\tau),\\
			4.&~ \theta(x, \tau)\theta_{0,-1/3}(x, \tau)
			=\theta_{0,-1/3}(2x, 2\tau)\theta_{0,1/3}(0, 2\tau)+\theta_{1/2,-1/3}(2x, 2\tau)\theta_{1/2,1/3}(0, 2\tau),\\
			5.&~ \theta(x, \tau)\theta_{1/3,0}(x, \tau)
			=\theta_{1/6,0}(2x, 2\tau)\theta_{-1/6,0}(0, 2\tau)+\theta_{2/3,0}(2x, 2\tau)\theta_{1/3,0}(0, 2\tau).\\
		\end{aligned}
	\end{equation*}
\end{lemma}

The classical null identity for Theta functions (see page 127, \cite{mckean1999elliptic}) is
\begin{equation}\label{null}
	\theta^4_{1/2,0}(0,\tau)+\theta^4_{0,1/2}(0,\tau)=\theta^4(0,\tau).
\end{equation}
This null identity is traditionally proved by considering the analytic properties of Theta functions including the order of zeros and periodicity and the quasi periodicity.
\par A new proof for the fourth-order classical identities of theta functions is given in the next lemma. In this proof, non-degeneracy of the eigenvalues and corresponding eigenvectors of the DFT for $n=2$ and elementary linear algebra is used. 
\begin{lemma} \label{classical}The following functional equations hold
	\begin{enumerate}
		\item $2^{nd}$ order :  
		$	\theta^2(x,2\tau)-\theta^2_{1/2,0}(x,2\tau)=\theta_{0,1/2}(x,\tau)\theta_{0,1/2}(0,\tau).$
		\item $4^{th}$ order classical theta function identities:
		\begin{itemize}
			\item $\theta^2(x,\tau)\theta^2(0,\tau)=\theta^2_{1/2,0}(x,\tau) \theta^2_{1/2,0}(0,\tau)+\theta^2_{0,1/2}(x,\tau) \theta^2_{0,1/2}(0,\tau),$
			\item $\theta^2(x,\tau)\theta^2_{1/2,0}(0,\tau)=\theta^2_{1/2,1/2}(x,\tau)\theta^2_{0,1/2}(0,\tau)+\theta^2_{1/2,0}(x,\tau)\theta^2_{0,1/2}(0,\tau).$
		\end{itemize}  
	\end{enumerate}
	
\end{lemma}
\begin{proof}
	To prove this lemma, consider the DFT for $n=2.$ Let $A$ be the DFT matrix of size 2, and $v_1, v_2$ be the eigenvectors corresponding to eigenvalues $1, -1$ respectively given by the Theorem \ref{matveev}. So, it is clear that 
	\begin{equation}
		A(v_1+v_2)= v_1-v_2.
	\end{equation} 
	Also,  
	\begin{equation*}
		A(v_1+v_2)= 2\sqrt{2}\begin{pmatrix}
			\theta(x, \tau)+\theta_{1/2,0}(x,\tau)\\
			\theta(x, \tau)-\theta_{1/2,0}(x,\tau)
		\end{pmatrix}, ~~\text{and} ~~
		v_1-v_2=2\sqrt{2} \begin{pmatrix}
			\theta (x/2,\tau/4)\\
			\theta_{0,1/2} (x/2,\tau/4)
		\end{pmatrix}.
	\end{equation*}
	\par Now equate the first and second entries of both sides and take their product, we will have,
	\begin{equation}\label{n1}
		\theta^2(x, \tau)-\theta^2_{1/2,0}(x, \tau)=\theta(x/2,\tau/4)\theta_{0,1/2} (x/2,\tau/4).
	\end{equation}
	Using identity (1) of lemma (\ref{Lemma1}) and equation (\ref{n1}), 
	\begin{equation}\label{New1}
		\theta^2(x,2\tau)-\theta^2_{1/2,0}(x,2\tau)=\theta_{0,1/2}(x,\tau)\theta_{0,1/2}(0,\tau).
	\end{equation}
	Evaluation at $x=0$ gives the following null identity:
	\begin{equation}\label{New2}
		\theta^2(0,2\tau)-\theta^2_{1/2,0}(0,2\tau)=\theta^2_{0,1/2}(0,\tau).
	\end{equation}
	The transformation of the equation (\ref{New1}) under $\tau \rightarrow \tau_1=\tau +1 $ using
	\begin{enumerate}
		\item $\theta(x,2\tau_1)=\theta(x,2\tau),$
		\item $\theta_{1/2,0}(x,2\tau_1)=i\theta_{1/2,0}(x,2\tau),$
		\item $\theta_{0,1/2}(x,\tau_1)=\theta(x,\tau).$
	\end{enumerate}
	is
	\begin{equation}\label{New3}
		\theta^2(x,2\tau)+\theta^2_{1/2,0}(x,2\tau)=\theta(x,\tau)\theta(0,\tau).
	\end{equation}
	From equations (\ref{New2}) and (\ref{New3}), we can write
	\begin{equation}\label{New4}
		\begin{split}
			&\theta^2(x,2\tau)\theta^2(0,2\tau)-\theta^2_{1/2,0}(x,2\tau) \theta^2_{1/2,0}(0,2\tau)\\
			&+ \theta^2_{1/2,0}(x,2\tau)\theta^2(0,2\tau)-\theta^2_{1/2,0}(0,2\tau)\theta^2(x,2\tau)\\
			=&\theta^2_{0,1/2}(0,\tau)\theta(x,\tau)\theta(0,\tau).
		\end{split}  
	\end{equation}
	Using Lemma \ref{Lemma1} we obtain,
	\begin{equation}\label{New5}
		\begin{split}
			&\theta^2(x,\tau)\theta^2(0,\tau)-\theta^2_{1/2,0}(x,\tau) \theta^2_{1/2,0}(0,\tau)-\theta^2_{0,1/2}(x,\tau) \theta^2_{0,1/2}(0,\tau)\\
			&= \theta^2(x,\tau)\theta^2_{1/2,0}(0,\tau)-\theta^2_{1/2,1/2}(x,\tau)\theta^2_{0,1/2}(0,\tau)-\theta^2_{1/2,0}(x,\tau)\theta^2_{0,1/2}(0,\tau).
		\end{split}  
	\end{equation}
	Now consider the transformation $\tau \rightarrow \tau_2=\tau+2.$
	\begin{enumerate}
		\item $\theta(x,\tau_2)=\theta(x,\tau), ~\text{and}~~\theta(0,\tau_2)=\theta(0,\tau),$ 
		\item $\theta_{1/2,0}(x,\tau_2)=i\theta_{1/2,0}(x,\tau), ~\text{and}~~\theta_{1/2,0}(0,\tau_2)=i\theta_{1/2,0}(0,\tau),$
		\item $\theta_{0,1/2}(x,\tau_2)=\theta_{0,1/2}(x,\tau),~\text{and}~~\theta_{0,1/2}(0,\tau_2)=\theta_{0,1/2}(0,\tau),$
		\item $\theta_{1/2,1/2}(x,\tau_2)=i\theta_{1/2,1/2}(x,\tau),~\text{and}~~\theta_{1/2,1/2}(0,\tau_2)=i\theta_{1/2,1/2}(0,\tau).$
	\end{enumerate}
	The equation (\ref{New5}) transformed to
	\begin{equation}\label{New6}
		\begin{split}
			&\theta^2(x,\tau)\theta^2(0,\tau)-\theta^2_{1/2,0}(x,\tau) \theta^2_{1/2,0}(0,\tau)-\theta^2_{0,1/2}(x,\tau) \theta^2_{0,1/2}(0,\tau)\\
			&= -\theta^2(x,\tau)\theta^2_{1/2,0}(0,\tau)+\theta^2_{1/2,1/2}(x,\tau)\theta^2_{0,1/2}(0,\tau)+\theta^2_{1/2,0}(x,\tau)\theta^2_{0,1/2}(0,\tau).
		\end{split}  
	\end{equation}
	The equations (\ref{New5}) and (\ref{New6}) are of the form $C=D$ and $C=-D$. Therefore $C=0,~\text{and}~ D=0.$
	The following 4th-degree equations are obtained
	\begin{equation}
		\begin{split}
			\theta^2(x,\tau)\theta^2(0,\tau)=&\theta^2_{1/2,0}(x,\tau) \theta^2_{1/2,0}(0,\tau)+\theta^2_{0,1/2}(x,\tau) \theta^2_{0,1/2}(0,\tau),\\
			\theta^2(x,\tau)\theta^2_{1/2,0}(0,\tau)=&\theta^2_{1/2,1/2}(x,\tau)\theta^2_{0,1/2}(0,\tau)+\theta^2_{1/2,0}(x,\tau)\theta^2_{0,1/2}(0,\tau).
		\end{split}
	\end{equation}
\end{proof}
\begin{remark}
	The classical null identities can be obtained by putting $x=0$ in the above lemma.
\end{remark}
Landen transformations are well studied in the literature. For example the Landen transformation in \cite{mckean1999elliptic}(page 150),in our notation is
\begin{align}\label{l10}
	\begin{aligned}
		\theta_{1/2,1/2}(2x,2\tau)&=\frac{\theta_{1/2,1/2}(0,\tau)\theta_{1/2,0}(x,\tau)}{\theta_{0,1/2}(0, 2\tau)},\\
		~~~~\text{and}~~~
		\theta^2(0,2\tau)&=\frac{1}{2}[\theta^2(0,\tau)+\theta^2_{0,1/2}(0, \tau)].
	\end{aligned}
\end{align}
Note that these identities involve doubling the quasi-period $\tau$. We now show that Landen transformations for even $n$ relate theta functions with quasi period $\tau$ to theta functions with quasi period $\frac{n^2}{2}\tau.$ For $n=2$, this involves doubling the quasi period. The following lemma justifies this assertion.
\begin{lemma}\label{ldn}
	Following Landen transformation holds 
	\begin{enumerate}
		\item For $n=3$
		\begin{equation*}
			\begin{split}
				&2\theta^2(3x, 9\tau)- [\theta_{1/3,0}(3x, 9\tau)+\theta_{-1/3,0}(3x, 9\tau) ]^2\\
				= &\theta_{0,1/3}(2x, 2\tau)\theta_{0,-1/3}(0, 2\tau)+\theta_{1/2,1/3}(2x, 2\tau)\theta_{1/2,-1/3}(0, 2\tau)\\
				&+ \theta_{0,-1/3}(2x, 2\tau)\theta_{0,1/3}(0, 2\tau)+\theta_{1/2,-1/3}(2x, 2\tau)\theta_{1/2,1/3}(0, 2\tau)\\
				&-\left[\theta_{1/6,0}(6x, 18\tau)\theta_{-1/6,0}(0, 18\tau)+\theta_{2/3,0}(6x, 18\tau)\theta_{1/3,0}(0, 18\tau)\right]\\
				&-\left[\theta_{-1/6,0}(6x, 18\tau)\theta_{1/6,0}(0, 18\tau)+\theta_{1/3,0}(6x, 18\tau)\theta_{2/3,0}(0, 18\tau)\right]
			\end{split}
		\end{equation*}	
		\item For $n=4$
		\begin{equation*}\label{l4}
			\begin{split}
				\big[ \theta(2x, 8\tau) + \theta_{1/2,0}(2x, 8\tau) \big]^2 - \big[ \theta_{1/4,0}(2x, 8\tau) + \theta_{3/4,0}(2x, 8\tau) \big]^2 
				=\theta_{0,1/2}(x, \tau)\theta_{0,1/2}(0, \tau).
			\end{split}		
		\end{equation*}
	\end{enumerate}
\end{lemma}
\begin{proof}
	\begin{enumerate}
		\item 	Let us consider the DFT matrix $A$ for $n=3,$ and the eigenvectors corresponding to eigenvalue $1$ and $-1$ are $v_1$ and $v_2$ given in Theorem \ref{matveev}. It is clear that 
		\begin{equation}\label{Landen1}
			A(v_1+v_2)=v_1-v_2.
		\end{equation}
		\begin{equation}\label{Landen2}
			v_1 +v_2 = 
			\begin{pmatrix} 4\theta(x, \tau) \\
				2 \big[\theta_{1/3,0}(x, \tau)+\theta_{-1/3,0}(x, \tau)\big] \\
				2 \big[\theta_{2/3,0}(x, \tau)+\theta_{-2/3,0}(x, \tau)\big]\\
			\end{pmatrix},
			v_1 - v_2= \begin{pmatrix}
				\frac{4}{\sqrt{3}} \theta\big(\frac{x}{3},\frac{\tau}{9} \big)\\
				\frac{2}{\sqrt{3}} \big[ \theta \big(\frac{x+1}{3},\frac{\tau}{9} \big) + \theta \big(\frac{x-1}{3},\frac{\tau}{9} \big) \big]\\
				\frac{2}{\sqrt{3}} \big[ \theta \big(\frac{x+1}{3},\frac{\tau}{9} \big) + \theta \big(\frac{x-1}{3},\frac{\tau}{9} \big) \big]\\
			\end{pmatrix}.
		\end{equation} 
		Let $\alpha = \theta_{1/3,0}(x, \tau)+\theta_{-1/3,0}(x, \tau) =\theta_{2/3,0}(x, \tau)+\theta_{-2/3,0}(x, \tau),$
		\begin{equation}\label{Landen3}
			A(v_1 + v_2)= \frac{1}{\sqrt{3}}\begin{pmatrix}
				4\theta(x, \tau)+4 \alpha \\
				4\theta(x, \tau)-2 \alpha \\
				4\theta(x, \tau)-2 \alpha \\
			\end{pmatrix},
		\end{equation}
		Use (\ref{Landen2}) and (\ref{Landen3}) in the equation (\ref{Landen1}), we have,
		\begin{equation*}
			\begin{split}
				\theta(x, \tau)+ \alpha&=\theta(x/3,\tau/9),\\
				2\theta(x, \tau)- \alpha&= \theta_{0,1/3} (x/3,\tau/9 ) + \theta_{0,-1/3}(x/3,\tau/9).
			\end{split}
		\end{equation*}
		This leads to 
		\begin{equation}\label{L2.1}
			\begin{aligned}
				&2\theta^2(x, \tau)- [\theta_{1/3,0}(x, \tau)+\theta_{-1/3,0}(x, \tau) ]^2+\theta(x,\tau)[\theta_{1/3,0}(x, \tau)+\theta_{-1/3,0}(x, \tau)]\\
				=&\theta(x/3,\tau/9)[\theta_{0,1/3} (x/3,\tau/9) + \theta_{0,-1/3}(x/3,\tau/9)].
			\end{aligned}
		\end{equation}
		Identities (3),(4) and (5) of lemma \ref{Lemma1} and equation (\ref{L2.1}) proves the result.
		\item For the proof consider the eigenvectors $v_1$ and $v_2$ corresponding to eigenvalues $1, -1$ of the DFT of size $4$. $A(v_1+v_2)=v_1-v_2.$ We now calculate $A(v_1+v_2)$ and $v_1-v_2$ separately, and equate first and third entries. 
		The identity is obtained by using result (1) of lemma \ref{Lemma1}  and by replacing $(x,\tau)$ by $(2x,8\tau)$. 
	\end{enumerate}
	
\end{proof}
Again by putting $x=0,$ in the above lemma, generalization of the null identities can be obtained.
\begin{remark}   For any even $n\geq 6$, the Landen transformation is 
	\begin{equation*}
		\begin{split}
			\left( \sum_{k=0}^{n/2-1} \theta_{2k/n,0} \bigg( \frac{nx}{2},\frac{n^2 \tau}{2}\bigg)\right)^2-\left( \sum_{k=0}^{n/2-1} \theta_{(2k+1)/n,0} \bigg( \frac{nx}{2},\frac{n^2 \tau}{2}\bigg)\right)^2
			=\theta_{1/2,0}(x, \tau)\theta_{1/2,0}(0, \tau).
		\end{split}
	\end{equation*}

	In identity (2) of lemma \ref{ldn} the Theta function with quasi period $\tau$ on the right-hand side is related to the Theta function with quasi period $\frac{n^2}{2} \tau = \frac{4^2}{2} \tau = 8\tau$ on the left-hand side. This is also true for the classical $n=2$ case where the Theta function with quasi period $\tau$ is related to the Theta function with quasi period $\frac{n^2}{2} \tau = \frac{2^2}{2} \tau = 2\tau$ (see eq (\ref{l10})). This gives a new insight into Landen transformations.
\end{remark}

\section{$q$ series}
Let $q=e^{\pi i \tau}$ with $\mid q \mid<1$. Consider the following infinite product \cite{zhai2009additive},
\begin{equation*}
	(\alpha:\beta)_{\infty}=\prod_{k=0}^{\infty}(1-\alpha \beta ^k),~~\text{and}  ~~ (\alpha:\beta)_{\infty}(\gamma:\beta)_{\infty} \cdots (\zeta:\beta)_{\infty} = (\alpha, \gamma, \cdots, \zeta:\beta)_{\infty}.
\end{equation*}

Also, recall the following infinite  product expressions for theta functions,
\begin{equation}\label{q2}
	\begin{split}
		\theta(x,\tau)=& (q^2, -qe^{2\pi ix}, -qe^{-2\pi ix} : q^2)_{\infty},\\
		\theta_{-1/2,1/2}(x,\tau)= & 2q^{1/4}\sin 
		\pi x(q^2,q^2e^{2\pi ix},q^2e^{-2\pi ix};q^2)_{\infty},\\
		\theta_{1/2,0}(x,\tau)=& 2q^{1/4}\cos \pi x(q^2,-q^2e^{2\pi ix},-q^2e^{-2\pi ix};q^2)_{\infty},\\
		\theta_{0,1/2}(x,\tau)=&(q^2, qe^{2\pi ix}, qe^{-2\pi ix} : q^2)_{\infty}
	\end{split}
\end{equation}
(See, for example Zhai \cite{zhai2009additive}, equations (13)-(16) and Whittaker and Watson \cite{whittaker1996course}, pp. 469-470)

\begin{lemma}\label{R-R}(\cite{bailey1951}, eq. no. (4.1))
	The Rogers-Ramanujan identity 
	\begin{equation*}
		\begin{aligned}
			\prod_{n=1}^{\infty} \left( 1-q^{4n}\right) \left( 1+q^{4n-3}z^2 \right) \left( 1+q^{4n-1}z^{-2} \right)& +  z \prod_{n=1}^{\infty} \left( 1-q^{4n}\right) \left( 1+q^{4n-1}z^2 \right) \left( 1+q^{4n-3}z^{-2} \right)\\
			&= \prod_{n=1}^{\infty} \left( 1-q^{n}\right) \left( 1+q^{n-1}z \right) \left( 1+q^{n}z^{-1} \right).
		\end{aligned}
	\end{equation*}
\end{lemma}
\begin{proof}
	Consider,
	\begin{equation*}
		\begin{aligned}
			\theta \big( x, \tau \big)+\theta_{0,\frac{1}{2}}(x, \tau)=& \sum_{m=-\infty}^{\infty}q^{m^2}e^{2\pi imx}+\sum_{m=-\infty}^{\infty}q^{m^2}e^{2\pi imx}(-1)^m,\\
			=& 2 \sum_{m=-\infty}^{\infty}q^{4m^2}e^{4\pi imx},
		\end{aligned}
	\end{equation*}
	
	Also consider the theta functions on the left side,
	
	\begin{equation*}
		\begin{aligned}
			\theta(4x, 16\tau)=&\left(q^{32}, -q^{16}e^{8\pi ix}, -q^{16}e^{-8\pi ix} : q^{32} \right)_{\infty},\\
			=& \left( q^{32}:q^{32}\right)_{\infty} \left( -q^{16}e^{8\pi ix}:q^{32}\right)_{\infty} \left( -q^{16}e^{-8\pi ix}:q^{32}\right)_{\infty},\\
			=& \prod_{m=0}^{\infty} \left( 1-q^{32m+32}\right) \left( 1+q^{32m+16}e^{8\pi ix} \right) \left( 1+q^{32m+16}e^{-8\pi ix} \right),\\
			& \text{replace $q^{4}$ by $q$ and $e^{4\pi i x}$ by $z$,  we have, }\\
			B=& \prod_{m=0}^{\infty} \left( 1-q^{8m+8}\right) \left( 1+q^{8m+4}z^2 \right) \left( 1+q^{8m+4}z^{-2} \right),\\
			& \text{replace $q^{2}$ by $q$ and $z$ by $z/q^{1/2}$,  we have, }\\
			B^{'}=& \prod_{m=0}^{\infty} \left( 1-q^{4m+4}\right) \left( 1+q^{4m+1}z^2 \right) \left( 1+q^{4m+3}z^{-2} \right),\\
			& \text{ put $m=n-1$,}\\
			B^{'}=&  \prod_{n=1}^{\infty} \left( 1-q^{4n}\right) \left( 1+q^{4n-3}z^2 \right) \left( 1+q^{4n-1}z^{-2} \right).
		\end{aligned}
	\end{equation*}	
	
	Also, 
	
	\begin{equation*}
		\begin{aligned}
			\theta_{1/2,0}(4x, 16\tau)=&2q^4\cos(4\pi x)\left(q^{32}, -q^{32}e^{8\pi ix}, -q^{32}e^{-8\pi ix} : q^{32} \right)_{\infty},\\
			=& q^4 (z+z^{-1})\left( q^{32}:q^{32}\right)_{\infty} \left( -q^{32}e^{8\pi ix}:q^{32}\right)_{\infty} \left( -q^{32}e^{-8\pi ix}:q^{32}\right)_{\infty},\\
			=&  q^4 (z+z^{-1})\prod_{m=0}^{\infty} \left( 1-q^{32m+32}\right) \left( 1+q^{32m+32}e^{8\pi ix} \right) \left( 1+q^{32m+32}e^{-8\pi ix} \right),\\
			& \text{replace $q^{4}$ by $q$ and $e^{4\pi i x}$ by $z$,  we have, }\\
			C=& q (z+z^{-1}) \prod_{m=0}^{\infty} \left( 1-q^{8m+8}\right) \left( 1+q^{8m+8}z^2 \right) \left( 1+q^{8m+8}z^{-2} \right),\\
			& \text{replace $q^{2}$ by $q$ and $z$ by $z/q^{1/2}$,  we have, }\\
			C^{'}=& q^{1/2} \left((z/q^{1/2})+(z/q^{1/2})^{-1}\right) \prod_{m=0}^{\infty} \left( 1-q^{4m+4}\right) \left( 1+q^{4m+3}z^2 \right) \left( 1+q^{4m+5}z^{-2} \right),\\
			=& z (1+qz^{-2}) \prod_{m=0}^{\infty} \left( 1-q^{4m+4}\right) \left( 1+q^{4m+3}z^2 \right) \left( 1+q^{4m+5}z^{-2} \right),\\
				\end{aligned}
		\end{equation*}	
			
			\begin{equation*}
				\begin{aligned}
			& \text{ put $m=n-1$,}\\
			C^{'}	=& z \prod_{n=1}^{\infty} \left( 1-q^{4n}\right) \left( 1+q^{4n-1}z^2 \right) \left( 1+q^{4n-3}z^{-2} \right).
		\end{aligned}
	\end{equation*}	
	
	By using same replacements in the left hand side of the first identity of lemma \ref{s3}, and by using Jacobi triple product identity
	
	\begin{equation*}
		\begin{aligned}
			A=& \sum_{m=-\infty}^{\infty}q^{m^2}z^{m},~~~(\text{by replacing $q^4$ by $q$ and putting $z$ in})\\
			=&  \prod_{n=1}^{\infty} \left( 1-q^{2n}\right) \left( 1+q^{2n-1}z \right) \left( 1+q^{2n-1}z^{-1} \right),\\
			& \text{replace $q^{2}$ by $q$ and $z$ by $z/q^{1/2}$,}\\
			A^{'}=& \prod_{n=1}^{\infty} \left( 1-q^{n}\right) \left( 1+q^{n-1}z \right) \left( 1+q^{n}z^{-1} \right).
		\end{aligned}
	\end{equation*}
	
	We have following "Rogers-Ramanujan type identity" given in (\cite{bailey1951}, eq. no. (4.1)):
	\begin{equation*}
		\begin{aligned}
			\prod_{n=1}^{\infty} \left( 1-q^{4n}\right) \left( 1+q^{4n-3}z^2 \right) \left( 1+q^{4n-1}z^{-2} \right)& +  z \prod_{n=1}^{\infty} \left( 1-q^{4n}\right) \left( 1+q^{4n-1}z^2 \right) \left( 1+q^{4n-3}z^{-2} \right)\\
			&= \prod_{n=1}^{\infty} \left( 1-q^{n}\right) \left( 1+q^{n-1}z \right) \left( 1+q^{n}z^{-1} \right).
		\end{aligned}
	\end{equation*}
\end{proof}
Evaluation first identity of lemma \ref{s3} at $x=0$, leads to
\begin{equation}\label{fq}
	\sum_{m=-\infty}^{\infty}q^{m^2}=\prod_{m=0}^{\infty}(1-q^{8m+8})\left[\prod_{m=0}^{\infty}(1+q^{8m+4})^2+2q\prod_{m=0}^{\infty}(1+q^{8m+8})^2\right].
\end{equation}
This identity representation of square numbers in the powers of $q$ advance by $8$.

We have the following result from the evaluation of the first identity of lemma \ref{s3} at $x=\tau$. This identity provides representation of odd square numbers in the powers of $q$ advance by $32$.
\begin{equation}\small
	\begin{split}
		\sum_{k=-\infty}^{\infty}q^{(2k+1)^2}
		=
		q\prod_{k=0}^{\infty}(1-q^{32k+32})(1+q^{32k+24})
		\left[\prod_{k=0}^{\infty}(1+q^{32k+8})+(1+q^8)\prod_{k=0}^{\infty}(1+q^{32k+40})\right].
	\end{split}
\end{equation}
 The evaluation of Landen transformation corresponding $n=4$ proved in lemma \ref{ldn}, at $x=0$  gives following $q$ identity is, 
\begin{equation}\label{fq}
	\begin{split}
		\sum_{m=-\infty}^{\infty}q^{m(2m+1)}
		=\frac{(q^{2};q^{2})^2_{\infty}[(-q;q^{2})^2_{\infty}+2q^{1/4}(-q^{2};q^{2})^2_{\infty}]^2-(q^{1/4};q^{1/4})^2_{\infty}(q^{1/8};q^{1/4})^4_{\infty}}{4q^{1/8}(q^{4};q^{4})_{\infty}[(-q^{2};q^{4})^2_{\infty}+2q^{1/2}(-q^{4};q^{4})^2_{\infty}]}.
	\end{split}
\end{equation}
This equation gives the triangular number identity for even-placed integers.
Additional $q$ identities can be similarly obtained.
\section{Concluding remarks}
In this article, we have shown that eigenvectors of the DFT, and elementary linear algebra can be exploited to obtain new linear, quadratic identities and Landen transformations among theta functions. The classical fourth-order identity is also a consequence of this elementary approach. The modular equations corresponding to new identities have been obtained.
\par There are corresponding identities that hold for generalized theta functions $\theta(x, \tau, \nu)$ for $\nu>1$. We have considered non-degenerate eigenvalues of the DFT (for $n\leq 8$). Some identities that involve degenerate eigenvalues of the DFT can also be considered.

	\bibliographystyle{plain}

\end{document}